\documentclass[11pt,a4paper]{amsart}
\usepackage[top=3cm,bottom=4cm,left=2.9cm,right=2.9cm,footskip=17mm]{geometry}
\usepackage{array,float,mathtools,verbatim}
\usepackage{amssymb,hyperref ,amsmath,amsthm,amsbsy,amscd, mathrsfs}
\usepackage{seqsplit}
\numberwithin{equation}{section}
\numberwithin{figure}{section}
\usepackage[shortlabels]{enumitem}
\usepackage{color}
\usepackage[normalem]{ulem}
\usepackage{tabularx}

\newtheorem{theorem}{Theorem}[section]

\newtheorem{lemma}[theorem]{Lemma}

\newtheorem{question}[theorem]{Question}

\newtheorem{conj}[theorem]{Conjecture}

\makeatletter
\newcommand{\imod}[1]{\allowbreak\mkern4mu({\operator@font mod}\,\,#1)}
\makeatother

\newcommand{\suchthat}{\;\ifnum\currentgrouptype=16 \middle\fi|\;}

\DeclareMathOperator{\rea}{Re}

\DeclareMathOperator{\Spec}{Spec}

\def \bfY {{\bf Y}}

\def\bb{\mathbb}

\def\bf{\mathbf}
\def\cal{\mathcal}

\def\ZZ{\bb{Z}}

\def\QQ{\bb{Q}}
\def\NN{\bb{N}}
\def\ZZ{\bb{Z}}

\def\Mat{\mathrm{Mat}} 
\def\rm{\textrm}

\def\a{\alpha}

\def\g{\mathfrak{g}} 
\def\p{\mathfrak{p}} 

\def\o{\mathfrak{o}} 
\def\O{\mathcal{O}} 

\def\G{\mathbf{G}} 
\def\H{\mathbf{H}}

\theoremstyle{definition}

\newtheorem{example}[theorem]{Example}
\newtheorem{remark}[theorem]{Remark}

\author{Duong Hoang Dung}
\address{Fakult\"{a}t f\"{u}r Mathematik\\
Universit\"{a}t Bielefeld\\
Postfach 100131\\
D-33501 Bielefeld, Germany.}
\email{dhoang@math.uni-bielefeld.de}


\date{\today}


\begin{document}
\title[Representation growth of Heisenberg group]{Representation growth of the Heisenberg group over $\O[x]/(x^n)$}

\begin{abstract}
We present a conjectured formula for the representation zeta function
of the Heisenberg group over $\O[x]/(x^n)$ where $\O$ is the ring of integers of some
number field. We confirm the conjecture for $n\leq 3$ and raise several questions.
\end{abstract}

\maketitle
\thispagestyle{empty}
\section{Introduction}
Let $G$ be a finitely generated torsion-free nilpotent group (or a $\cal{T}$-group for short). Two complex representations $\rho$ and $\sigma$ of $G$ are called \textit{twist-equivalent} if there exists a $1$-dimensional representation $\lambda$ of $G$ such that $\rho=\lambda\otimes\sigma$. Twist-equivalence is an equivalence relation on the set of finite dimensional
irreducible complex representations of $G$ and its classes are called
\textit{twist-isoclasses}. The numbers $r_n(G)$ of twist-isoclasses of 
dimension $n$ are finite for all $n$, cf.~{\cite[Theorem 6.6]{Lubotzky-Magid}}. The \textit{representation zeta function} of $G$ is defined to be the Dirichlet
generating function
$$\zeta_G(s):=\sum_{n=1}^\infty\frac{r_n(G)}{n^s},$$
where $s$ is a complex variable.
The sequence $(r_n(G))$ grows polynomially and thus $\zeta_G(s)$ converges on 
a complex half-plane $\rea(s)>\alpha$, cf.~{\cite[Lemma 2.1]{SV/14}}. The 
infimum of such $\a$ is the \textit{abscissa of convergence} $\a(G)$ of 
$\zeta_G(s)$ which gives the precise degree of polynomial growth; i.e., $\a(G)$ is the smallest value such that 
$\sum_{n=1}^Nr_n(G)=O(N^{\a(G)+\epsilon})$ for every $\epsilon\in\bb{R}_{>0}$.

Let $\H$ be the Heisenberg group scheme associated to the Heisenberg $\ZZ$-Lie lattice of strict upper-triangular $3\times 3$ matrices. For every ring $R$, the group $\H(R)$ is isomorphic to the group of upper-unitriangular $3\times3$ matrices over $R$. If $R$  is torsion-free finitely
generated over $\ZZ$, then $\H(R)$ is a $\cal{T}$-group
 of nilpotency class $2$ and Hirsch length $3\cdot\rm{rk}_{\ZZ}(R)$.
When $R=\O$ is the ring of integers of a number field $K$, the zeta function
of $\H(\O)$ is
\begin{equation}\label{H over O}
\zeta_{\H(\O)}(s)=\frac{\zeta_K(s-1)}{\zeta_K(s)}=\prod_{\p\in\Spec(\O)}
\frac{1-|\O/\p|^{-s}}{1-|\O/\p|^{1-s}},
\end{equation}
where $\zeta_K(s)$ is the Dedekind zeta function of $K$, $\p$
ranges over the nonzero prime ideals of $\O$. This is proved in {\cite{Magid89}}
for $K=\QQ$, in \cite{Ezzat} for quadratic number fields, and 
in {\cite[Theorem~B]{SV/14}} for arbitrary number fields.
The zeta function $\zeta_{\H(\O)}(s)$ has abscissa of convergence $\a(\H(\O))=2$, which is independent of $K$, and may be meromorphically continued to the whole
complex plane.

In this paper, we consider the Heisenberg group over rings of the form $\O[x]/(x^n)$. If $n=1$ then it is the Heisenberg group over $\O$. The zeta function of
 $\H(\O[x]/(x^2))$  was computed in {\cite[Example 6.5]{Snocken}}. We compute  
the zeta function of $\H(\O[x]/(x^n))$ for $n=3$.

\subsection*{Organization and notation}
In Section 2, we recall  formulae of local representation
zeta functions in terms of $p$-adic integrals. The zeta function for the case $n=3$ is computed
in Section 3. We conclude in Section 4 with several questions and
conjectures.

\subsection*{Acknowledgements} We acknowledge support from the DFG Sonderforschungsbereich 701 at Bielefeld University. We thank Tobias Rossmann and Christopher
Voll for several helpful discussions.

\section{Preliminaries}
\subsection{Local representation zeta functions}
The group $\H(\O[x]/(x^n))$ is a $\cal{T}$-group of nilpotency class $2$ and 
Hirsch length $3n\cdot\rm{rk}_\ZZ(\O)$. The zeta function $\zeta_{\H(\O[x]/(x^n))}(s)$
has an Euler factorization (cf.~{\cite[Proposition 2.2]{SV/14}})
$$\zeta_{\H(\O[x]/(x^n))}(s)=\prod_{\p\in\Spec(\O)}\zeta_{\H(\O_\p[x]/(x^n))}(s),$$
where $\p$ ranges over the nonzero prime ideals in $\O$ and $\O_\p$ is the  completion of $\O$ at  $\p$.
The local factors $\zeta_{\H(\O_\p[x]/(x^n))}$ are rational in $|\O/\p|^{-s}$ and almost all of them
satisfy a functional equation (cf.~{\cite[Theorem~A]{SV/14}}).

The $\O$-lattice associated to $\bf{H}(\O[x]/(x^n))$ has the following presentation; see {\cite[Section~2.4]{SV/14}}:
$$\left\langle 
\begin{matrix}
x_0,x_1,\cdots,x_{n-1} \\
y_0,y_1,\cdots,y_{n-1} \\
z_0,z_1,\cdots,z_{n-1}
\end{matrix}
\suchthat[x_i,y_j]=
\left\{
\begin{array}{ll}z_{i+j}&\rm{if}~i+j<n,\\
0&\rm{otherwise}.
\end{array}
\right.
\right\rangle.$$
The associated commutator matrix with respect to the chosen $\O$-basis is defined by
$$\cal{R}_n(\bf{Y})=\left(
\begin{array}{c|c}

0&Q_n(\bf{Y})\\ \hline
-Q_n(\bf{Y})^t&0
\end{array}
\right),
$$
where
$$Q_n(\bf{Y})=\left(
\begin{array}{cccccc}
Y_1&Y_2&Y_3&\cdots&Y_{n-1} & Y_{n}\\
Y_2&Y_3&\cdots&&Y_{n}&0 \\
Y_3&\cdots&&\cdots&&0\\
 \vdots& &\cdots&&&\vdots\\
Y_{n-1} &Y_{n} &\cdots&\cdots&\cdots&\vdots\\
Y_{n}&0&\cdots&\cdots&\cdots&0
\end{array}
\right)\in\Mat_{n}(\O[Y_1,\cdots,Y_n]).$$

Fix a nonzero prime ideal $\p$ and denote $\o:=\O_\p$. Let $q:=|\o/\p|$ be the residue field cardinality and $p$ its characteristic. Let $W_n(\o)=\o^n\setminus \p^n$. Set
\begin{equation}\label{p-adic int}
\cal{Z}_\p(\rho,\tau):=
\int_{(u,{\bf y})\in \p\times W_n(\o)} |u|^\tau
\prod_{j=1}^n\frac{\|{F}_j({\bf y})\cup {F}_{j-1}({\bf y})u^2\|^\rho}
{\|{F}_{j-1}({\bf y})\|^\rho}d\mu,
\end{equation}
where the additive Haar measure $\mu$ on $\o^{n+1}$ is normalized such that $\mu(\o^{n+1})=1$, and
\begin{align*}
      {F}_j(\bfY) & = \{ f\mid f=f(\bfY)~\rm{a principal}~2j\times
   2j ~\rm{minor of}
   ~\cal{R}_n(\bfY) \},\label{eq:poly F}\\
   \|H(X,\bfY)\| & = \min\{v_p(h(X,\bfY))\mid h\in H\}
   ~\rm{for a finite set}~H\subset\o[X,\bfY].\nonumber
 \end{align*} 
The local factor $\zeta_{\H(\o[x]/(x^n))}(s)$ can be expressed in terms of the $p$-adic integral \eqref{p-adic int} as the following (cf.{\cite[Corollary 2.11]{SV/14}}):
\begin{equation}
\zeta_{\H(\o[x]/(x^n))}=1+(1-q^{-1})^{-1}\cal{Z}_\p(-s/2,ns-n-1).
\end{equation}

\subsection{Auxiliary lemmas}

\begin{lemma}\label{Lem identities}
The following identities hold in the field of formal Laurent series  $\QQ((a,b,c))$.
\begin{enumerate}
\item $\sum_{(X,Y)\in\NN^2}a^Xb^Yc^{\min\{X,Y\}}=\frac{abc(1-ab)}{(1-abc)(1-a)(1-b)}$. \label{Lem SV15}
\item $\sum_{(X,Y)\in\NN^2}a^Xb^Yc^{\min\{X,2Y\}}=\frac{abc(1-a+ac-a^2bc)}{(1-a)(1-b)(1-a^2bc^2)}$.
\item $\sum_{(X,Y)\in\NN^2}a^Xb^Yc^{\min\{X,Y\}}c^{\min\{X,2Y\}}=\frac{abc^2(1-a+ac-abc-a^2bc^3+a^3b^2c^3)}{(1-a)(1-b)(1-abc^2)(1-a^2bc^3)}$.
\item 
\begin{align*}
&\sum_{(X,Y,Z)\in\NN^3}a^Xb^Yc^Zd^{\min\{X,Y+2Z\}}d^{\min\{X,2Y+4Z\}} 
=\frac{abd^2}{1-abd^2}\frac{1}{1-b}\frac{c}{1-c}+\\
&+ \frac{acd^2}{1-acd^2}\frac{abd^2}{1-abd^2}\frac{1}{1-c}
+ \frac{a^2cd^4}{1-a^2cd^4}\frac{abd^2}{1-abd^2}\frac{1}{1-acd^2}
+\frac{a^2bd^3}{1-a^2bd^3}\frac{1}{1-abd^2}\frac{a^2cd^4}{1-a^2cd^4}\\
&+\frac{a^2bd^3}{1-a^2bd^3}\frac{1-a+ad-a^4cd^5}{(1-a)(1-a^2cd^4)(1-a^4cd^6)}. 
\end{align*}
\end{enumerate}
\end{lemma}
\begin{proof}
The identity $(1)$ is from {\cite[Lemma 2.2]{SV/15}}\label{Lem SV15}. 
We present the proofs of (2) and (3) while (4) is proven
similarly.

For (2), consider the case $X\leq Y$ and let $Y=X+Y'$ with $Y'\in\NN_0$. Then
$$\sum_{\substack{(X,Y)\in\NN^2\\ X\leq Y}}a^Xb^Yc^{\min\{X,2Y\}}=
\sum_{(X,Y')\in\NN\times\NN_0}a^Xb^{X+Y'}c^X=\frac{abc}{1-abc}\frac{1}{1-b}.$$
Consider the case $X>Y$ and let $X=Y+X'$ with $X'\in\NN$. Then
\begin{align*}
\sum_{\substack{(X,Y)\in\NN^2\\ X> Y}}a^Xb^Yc^{\min\{X,2Y\}}
&=\sum_{(X',Y)\in\NN^2}a^{X'+Y}b^Yc^{\min\{X'+Y,2Y\}}\\
&=\sum_{(X',Y)\in\NN^2}a^{X'}(abc)^Yc^{\min\{X',Y\}}\\
&=\frac{a^2bc^2(1-a^2bc)}{(1-a^2bc^2)(1-a)(1-abc)}
\end{align*}
by (1). Hence
\begin{align*}
\sum_{(X,Y)\in\NN^2}a^Xb^Yc^{\min\{X,2Y\}} &=\frac{abc}{1-abc}\frac{1}{1-b}+\frac{a^2bc^2(1-a^2bc)}{(1-a^2bc^2)(1-a)(1-abc)}\\
&=\frac{abc(1-a+ac-a^2bc)}{(1-a)(1-b)(1-a^2bc^2)}.
\end{align*}

For (3), first consider the case $X\leq Y$ and let $Y=X+Y'$ with $Y'\in\NN_0$. Then
$$\sum_{\substack{(X,Y)\in\NN^2\\ X\leq Y}} a^Xb^Yc^{\min\{X,Y\}}c^{\min\{X,2Y\}}=
\sum_{(X,Y')\in\NN\times\NN_0}a^Xb^{X+Y'}c^Xc^X=\frac{abc^2}{1-abc^2}\frac{1}{1-b}.$$
Consider now the case $X>Y$ and let $X=X'+Y$ with $X'\in\NN$. Then, by (1)
\begin{align*}
\sum_{\substack{(X,Y)\in\NN^2\\ X\leq Y}} a^Xb^Yc^{\min\{X,Y\}}c^{\min\{X,2Y\}}&=
\sum_{(X',Y)\in\NN^2}a^{X'+Y}b^Yc^Yc^{\min\{X'+Y,2Y\}}\\
&= \sum_{(X',Y)\in\NN^2}a^{X'}(abc^2)^Yc^{\min\{X',Y\}}\\
&=\frac{a^2bc^3(1-a^2bc^2)}{(1-a^2bc^3)(1-a)(1-abc^2)}.
\end{align*}
Hence
\begin{align*}
\sum_{(X,Y)\in\NN^2}a^Xb^Yc^{\min\{X,Y\}}c^{\min\{X,2Y\}} & =\frac{abc^2}{1-abc^2}\frac{1}{1-b}+\frac{a^2bc^3(1-a^2bc^2)}{(1-a^2bc^3)(1-a)(1-abc^2)}\\
&=\frac{abc^2(1-a+ac-abc-a^2bc^3+a^3b^2c^3)}{(1-a)(1-b)(1-abc^2)(1-a^2bc^3)}.
\end{align*}
\end{proof}

\begin{lemma}\label{Z_2}
$$I:=\int_{x,y\in\p}|x|^{3s-4}\|x,y^3\|^{-s}d\mu
=(1-q^{-1})\frac{q^2t^2(1+q^3t^2-q^3t^3+q^6t^4-q^6t^5-q^8t^7)}{(1-q^3t^3)(1-q^8t^6)},$$
where $t:=q^{-s}$.
\end{lemma}
\begin{proof}
Write $I=I_1+I_2$ where
\begin{align*}
I_1 &:= \int_{\substack{x,y\in\p\\ v(x)\leq v(y)}}|x|^{3s-4}\|x,y^3\|^{-s}d\mu,\\
I_2 &:=\int_{\substack{x,y\in\p\\ v(x)>v(y)}}|x|^{3s-4}\|x,y^3\|^{-s}d\mu.
\end{align*}
\textit{Computation of $I_1$.} Let $y=xy_1$ with $y_1\in\o$. Then
$$I_1=\int_{\substack{x\in \p\\ y_1\in\o}}
|x|^{3s-4}|x|^{-s+1}d\mu =(1-q^{-1})\frac{q^2t^2}{1-q^2t^2}.$$
\textit{Computation of $I_2$.} Let $x=yx_1$ with $x_1\in\p$. Then
$$I_2=\int_{x_1,y\in\p}|x_1|^{3s-4}|y|^{2s-3}\|x_1,y^2\|^{-s}d\mu.$$
Since $\mu(\{(x_1,y)\in\p^2\mid v(x_1)=X, v(y)=Y\})=(1-q^{-1})^2q^{-X-Y}$, one has by Lemma~\ref{Lem identities}~(2)
\begin{align*}
I_2&=(1-q^{-1})^2\sum_{(X,Y)\in\NN^2}q^{-X-Y}q^{(-3s+4)X}q^{(-2s+3)Y}q^{s\min\{X,2Y\}}\\
&=(1-q^{-1})^2\sum_{(X,Y)\in\NN^2}q^{(-3s+3)X}q^{(-2s+2)Y}q^{s\min\{X,2Y\}}\\
&=(1-q^{-1})^2\frac{q^{-3s+3-2s+2+s}(1-q^{-3s+3}+q^{-3s+3+s}-q^{-6s+6-2s+2+s})}{(1-q^{-3s+3})(1-q^{-2s+2})(1-q^{-6s+6-2s+2+2s})} \\
&=(1-q^{-1})^2\frac{q^5t^4(1-q^3t^3+q^3t^2-q^8t^7)}{(1-q^2t^2)(1-q^3t^3)(1-q^8t^6)}.
\end{align*}
Hence
$$I=I_1+I_2=(1-q^{-1})\frac{q^2t^2(1+q^3t^2-q^3t^3+q^6t^4-q^6t^5-q^8t^7)}{(1-q^3t^3)(1-q^8t^6)}.$$
\end{proof}

\subsection{The zeta function of $\H(\O[x]/(x^2))$} 
In this case $\tau=vs-2-1=2s-3$ and
$F_0(\bf{Y})=1, F_1(\bf{Y})=\{X^2,Y^2\}$, $F_2(\bf{Y})=\{Y^4\}$. We have
$$\cal{Z}:=\cal{Z}_\p(-s/2,2s-3)=\int_{\substack{u\in \p\\ \bf{y}=(x,y)\in W_2(\o)}}|u|^{2s-3}\|u,x,y\|^{-s}\|ux,uy,y^2\|^{-s}d\mu.$$
It is computed in {\cite[Example 6.5]{Snocken}} that
$$\zeta_{\H(\o[x]/(x^2))}(s)=1+(1-q^{-1})^{-1}\cal{Z}
= \frac{(1-t)(1-q^2t^2)}{(1-qt)(1-q^3t^2)},~\rm{where}~t:=q^{-s}.$$
Hence
\begin{equation}\label{n=2}
\zeta_{\H(\O[x]/(x^2))}(s)=\frac{\zeta_K(s-1)\zeta_K(2s-3)}{\zeta_K(s)\zeta_K(2s-2)}.
\end{equation}

\section{The zeta function of $\H(\O[x]/(x^3))$}
 In this case, $\tau=3s-4$ and
 
\begin{align*}
F_0(\bf{Y}) &=1,\\
F_1(\bf{Y}) &= \{X^2,Y^2,Z^2\},\\
F_2(\bf{Y}) &= \{Z^4,Y^2Z^2,(XZ-Y^2)^2\},\\
F_3(\bf{Y}) &=\{Z^6\}.
\end{align*}

Set
\begin{align*}
A &:=\|z^2,yz,xz-y^2\|^s,\\
B &:= \|z^2,yz,xz-y^2,xu,yu,uz\|^{-s},\\
C &:= \|z^3,z^2u,yzu,(xz-y^2)u\|^{-s}.
\end{align*}
Then 
$$\zeta_{\H(\o[x]/(x^3))}(s)=1+(1-q^{-1})^{-1}\cal{Z}(s),$$
where
$$\cal{Z}:=\cal{Z}_\p(-s/2,3s-4)=\int_{u\in\p,\bf{y}\in W_3(\o)}
|u|^{3s-4}ABCd\mu.$$
Write $\cal{Z}=\cal{Z}_1+\cal{Z}_2+\cal{Z}_3$ where
\begin{align*}
\cal{Z}_1 &:= \int_{\substack{u\in\p,z\in W_1(\o)\\ x,y\in\o}}|u|^{3s-4}ABCd\mu,\\
\cal{Z}_2 &:= \int_{\substack{u,z\in\p\\ x\in\o \\ y\in W_1(\o)}}|u|^{3s-4}ABCd\mu,\\
\cal{Z}_3 &:= \int_{\substack{u,y,z\in\p\\ x\in W_1(\o)}}|u|^{3s-4}ABCd\mu.
\end{align*}

\subsection{Computation of $\cal{Z}_1$}
Since $z\in W_1(\o)$, it follows that $A=B=C=1$. Hence
$$\cal{Z}_1=\int_{\substack{u\in\p,z\in W_1(\o)\\ x,y\in\o}}|u|^{3s-4}d\mu
=(1-q^{-1})^2\frac{q^3t^3}{1-q^3t^3}.$$
\subsection{Computation of $\cal{Z}_2$}
Since $y\in W_1(\o)$ and $z\in \p$, it follows that $xz-y\in W_1(\o)$ and
$A=1, ~B=1,~C=\|u,z^3\|^{-s}$. Thus
$$\cal{Z}_2=(1-q^{-1})\int_{u,z\in\p}|u|^{3s-4}\|u,z^3\|^{-s}d\mu.$$
It follows from Lemma~\ref{Z_2} that
$$\cal{Z}_2=(1-q^{-1})^2\frac{q^2t^2(1+q^3t^2-q^3t^3+q^6t^4-q^6t^5-q^8t^7)}{(1-q^3t^3)(1-q^8t^6)}.$$
\subsection{Computation of $\cal{Z}_3$} In this case, $x\in W_1(\o), y,z\in\p$, whence
\begin{align*}
A &= \|z^2,yz,xz-y^2\|^s,\\
B &= \|u,z^2,yz,xz-y^2\|^{-s},\\
C &= \|z^3,z^2u,yzu,(xz-y)^2u\|^{-s}.
\end{align*}
Write $\cal{Z}_3=\cal{Z}_{31}+\cal{Z}_{32}$, where
\begin{align*}
\cal{Z}_{31}&:= \int_{\substack{u,y,z\in\p\\ x\in W_1(\o)\\ v(y)>v(z)}}|u|^{3s-4}ABCd\mu, \\
\cal{Z}_{32}&:=\int_{\substack{u,y,z\in\p\\ x\in W_1(\o)\\ v(y)\leq v(z)}}|u|^{3s-4}ABCd\mu.
\end{align*}
\subsubsection{Computation of $\cal{Z}_{31}$} Let $y=zy_1$ with $y_1\in\p$. Then
\begin{align*}
A &= \|z^2,xz-y_1^2z^2\|^s = |z|^s\|z,x-y_1^2z\|^s=|z|^s~(\rm{since}~x-y_1^2z\in W_1(\o)),\\
B &= \|u,z\|^{-s},\\
C &= |z|^{-s}\|u,z^2\|^{-s}.
\end{align*}
Thus
$$
\cal{Z}_{31} = q^{-1}(1-q^{-1})\int_{u,z\in \p}|z||u|^{3s-4}\|u,z\|^{-s}\|u,z^2\|^{-s}d\mu.$$
Since $\mu\{(u,z)\in\p\mid v(u)=X, v(z)=Y \}=(1-q^{-1})^2q^{-X-Y}$, Lemma~\ref{Lem identities}~(3) implies that
\begin{align*}
\cal{Z}_{31} &= q^{-1}(1-q^{-1})\int_{u,z\in \p}|u|^{3s-4}|z|\|u,z\|^{-s}\|u,z^2\|^{-s}d\mu \\
&= q^{-1}(1-q^{-1})^3\sum_{(X,Y)\in\NN^2}q^{-X-Y}q^{(-3s+4)X}q^{-Y}q^{s\min\{X,Y\}}q^{s\min\{X,2Y\}}\\
&=q^{-1}(1-q^{-1})^3\sum_{(X,Y)\in\NN^2}q^{(-3s+3)X}q^{-2Y}q^{s\min\{X,Y\}}q^{s\min\{X,2Y\}}\\
&=(1-q^{-1})^3\frac{t(1-q^3t^3+q^3t^2-qt^2-q^4t^3+q^5t^6)}{(1-q^{-2})(1-qt)(1-q^3t^3)(1-q^4t^3)}. 
\end{align*}

\subsubsection{Computation of $\cal{Z}_{32}$} Let $z=yz_1$ with $z_1\in\o$. We have
\begin{align*}
A &= |y|^s\underbrace{\|yz_1,xz_1-y\|^s}_{:=A_1},\\
B &= \|u,y^2z_1,y(xz_1-y)\|^{-s},\\
C &= |y|^{-s}\underbrace{\|y^2z_1^3,yz_1u,u(xz_1-y)\|^{-s}}_{:=C_1}.
\end{align*}
Thus
$$\cal{Z}_{32}=\int_{\substack{x\in W_1(\o)\\ u,y\in\p \\ z_1\in\o}}
|u|^{3s-4}|y|A_1BC_1d\mu.$$
Write $\cal{Z}_{32}=\cal{Z}_{321}+\cal{Z}_{322}$, where
\begin{align*}
\cal{Z}_{321}&:=\int_{\substack{x\in W_1(\o)\\ u,y\in\p \\ z_1\in W_1(\o)}}
|u|^{3s-4}|y|A_1BC_1d\mu,\\
\cal{Z}_{322}&:= \int_{\substack{x\in W_1(\o)\\ u,y\in\p \\ z_1\in\p}}
|u|^{3s-4}|y|A_1BC_1d\mu.
\end{align*}
\textit{Computation of $\cal{Z}_{321}$}. Since $z_1\in W_1(\o)$, it follows that $xz_1-y\in W_1(\o)$ and so
$$\cal{Z}_{321} =(1-q^{-1})^2\int_{u,y\in\p}|u|^{3s-4}|y|\|u,y\|^{-s}\|u,y^2\|^{-s}d\mu =(q-1)\cal{Z}_{31}.$$
\textit{Computation of $\cal{Z}_{322}$}. Write $\cal{Z}_{322}=\cal{Z}_{322a}+\cal{Z}_{322b}$, where
\begin{align*}
\cal{Z}_{322a}&:= \int_{\substack{x\in W_1(\o)\\ u,y,z_1\in\p \\ v(y)>v(z_1)}}
|u|^{3s-4}|y|A_1BC_1d\mu,\\
\cal{Z}_{322b}&:= \int_{\substack{x\in W_1(\o)\\ u,y,z_1\in\p \\ v(y)\leq v(z_1)}}
|u|^{3s-4}|y|A_1BC_1d\mu.
\end{align*}

\noindent\textit{Computation of $\cal{Z}_{322a}$.} Let $y=z_1y_1$ with $y_1\in\p$. We have
\begin{align*}
A_1 &= |z_1|^s~(\rm{since}~x-y_1\in W_1(\o)),\\
B  &=\|u,y_1z_1^2\|^{-s},\\
C_1 &= |z_1|^{-s}\underbrace{\|y_1^2z_1^4,u\|^{-s}}_{:=C_2}.
\end{align*}
Thus
$$
\cal{Z}_{322a}=(1-q^{-1})\int_{u,y_1,z_1\in\p}|u|^{3s-4}|y_1||z_1|^2BC_2d\mu.$$
Since $\mu\{(u,y_1,z_1)\in\p\mid v(u)=X, v(y_1)=Y, v(z_1)=Z \}
=(1-q^{-1})^3q^{-X-Y-Z}$, one has
\begin{align*}
\cal{Z}_{322a}&=(1-q^{-1})^4\sum_{(X,Y,Z)\in\NN^3}q^{-X-Y-Z}
q^{(-3s+4)X}q^{-Y}q^{-2Z}q^{s\min\{X,Y+2Z\}}q^{s\min\{X,2Y+4Z\}}\\
&=(1-q^{-1})^4\sum_{(X,Y,Z)\in\NN^3}q^{(-3s+3)X}q^{-2Y}q^{-3Z}
q^{s\min\{X,Y+2Z\}}q^{s\min\{X,2Y+4Z\}}.
\end{align*}
One now can apply Lemma~\ref{Lem identities}~(4) with $a=q^{-3s+3},~
b=q^{-2},~c=q^{-3}$ and $d=q^{s}$ to obtain $\cal{Z}_{322a}$.
We record the result in the Appendix.

\noindent\textit{Computation of $\cal{Z}_{322b}$.} Let $z_1=yz_2$ with $z_2\in\o$. We have
\begin{align*}
A_1 &= |y|^s\underbrace{\|yz_2,xz_2-1\|^s}_{:=A_2},\\
B &= \|u,y^3z_2,y^2(xz_2-1)\|^{-s},\\
C_1 &= |y|^{-s}\underbrace{\|y^4z_2^3,yz_2u,u(xz_2-1)\|^{-s}}_{:=C_2}.
\end{align*}
Thus
$$\cal{Z}_{322b}=\int_{\substack{x\in W_1(\o)\\ z_2\in\o \\ u,y\in\p}}
|u|^{3s-4}|y|^2 A_2BC_2d\mu=\cal{Z}_{322b1}+\cal{Z}_{322b2},$$
where
\begin{align*}
\cal{Z}_{322b1} &:= \int_{\substack{x\in W_1(\o)\\ z_2\in\p \\ u,y\in\p}}
|u|^{3s-4}|y|^2 A_2BC_2d\mu,\\
\cal{Z}_{322b2} &:=\int_{\substack{x\in W_1(\o)\\ z_2\in W_1(\o) \\ u,y\in\p}}
|u|^{3s-4}|y|^2 A_2BC_2d\mu.
\end{align*}
\noindent\textit{Computation of $\cal{Z}_{322b1}$.} Since $z_2\in\p$, it follows that $xz_2-1\in W_1(\o)$. Thus $A_2=1, B=\|u,y^2\|^{-s}$ and $C_2=\|y^4z_2^3,u\|^{-s}$. It's now easy to compute
$$\cal{Z}_{322b1}=(1-q^{-1})\int_{u,y,z_2\in\p}|u|^{3s-4}|y|^2 BC_2d\mu.$$
Since $\mu\{(u,y,z_2)\in\p^3\mid v(u)=X,~v(y)=Y,~v(z_2)=Z\}=(1-q^{-1})^3q^{-X-Y-Z}$, one has
\begin{align*}
\cal{Z}_{322b1}&=(1-q^{-1})^4\sum_{(X,Y,Z)\in\NN^3}q^{-X-Y-Z}q^{(-3s+4)X}q^{-2Y}q^{s\min\{X,2Y\}}q^{s\min\{X,4Y+3Z\}}\\
&=(1-q^{-1})^4\sum_{(X,Y,Z)\in\NN^3}q^{(-3s+3)X}q^{-3Y}q^{-Z}q^{s\min\{X,2Y\}}q^{s\min\{X,4Y+3Z\}}.
\end{align*}
One needs first to compute $\sum_{(X,Y,Z)\in\NN^3}a^Xb^Yc^Zd^{\min\{X,2Y\}}d^{\min\{X,4Y+3Z\}}$
similarly to Lemma~\ref{Lem identities}~(4) and then apply to $a=q^{-3s+3},~b=q^{-3},~c=q^{-1}$ and $d=q^s$ to
obtain $\cal{Z}_{322b1}$. The result is recorded in the Appendix.\\
\textit{Computation of $\cal{Z}_{322b2}$.} The equation $xz_2\equiv 1\!\mod \p$ has $q-1$ roots $(a_1,a_2)\in (\mathbb{F}_q^*)^2$. We have
\begin{align*}
\cal{Z}_{322b2} &= \int_{\substack{x,z_2\in W_1(\o)\\ u,y\in\p}}
|u|^{3s-4}|y|^2 A_2BC_2d\mu \\
&=\sum_{(a_1,a_2)\in(\mathbb{F}_q^*)^2}\int_{\substack{(x,z_2)\in(a_1,a_2)+\p^2\\ u,y\in \p}}
|u|^{3s-4}|y|^2 A_2BC_2d\mu \\
& = (q-1)(q-2)J_1+(q-1)J_2,
\end{align*}
where
\begin{align*}
J_1 &:=\int_{\substack{(x,z_2)\in(a_1,a_2)+\p^2\\ a_1a_2\not\equiv 1\!\mod\p\\ u,y\in\p}}|u|^{3s-4}|y|^2 A_2BC_2d\mu,\\
J_2 &:=\int_{\substack{(x,z_2)\in(a_1,a_2)+\p^2\\ a_1a_2\equiv 1\!\mod\p\\ u,y\in\p}}|u|^{3s-4}|y|^2 A_2BC_2d\mu.
\end{align*}

In computing $J_1$, notice that in this case $xz_1\not\equiv\! 1\mod\p$, and so
$A_2=1, B=\|u,y^2\|^{-s}$ and $C_2=\|u,y^4\|^{-s}$, and thus we have
$$J_1=q^{-2}\int_{u,y\in\p}|u|^{3s-4}|y|^2BC_2d\mu.$$
Since $\mu\{(u,y)\in\p^2\mid v(u)=X,~v(y)=Y\}=(1-q^{-1})q^{-X-Y}$, one has
\begin{align*}
J_1&=q^{-2}(1-q^{-1})^2\sum_{(X,Y)\in\NN^2}q^{-X-Y}q^{(-3s+4)X}q^{-2Y}q^{s\min\{X,2Y\}}q^{s\min\{X,4Y\}}\\
&=q^{-2}(1-q^{-1})^2\sum_{(X,Y)\in\NN^2}q^{(-3s+3)X}q^{-3Y}q^{s\min\{X,2Y\}}q^{s\min\{X,4Y\}}.
\end{align*}
We first need to compute $\sum_{(X,Y)\in\NN^2}a^Xb^Yc^{\min\{X,2Y\}}c^{\min\{X,4Y\}}$ similarly to 
Lemma~\ref{Lem identities}~(3) and then apply with $a=q^{-3s+3},~b=q^{-3}$ and $c=q^{s}$ to obtain $J_1$.
We record $J_1$ in the Appendix.\\
In computing $J_2$, notice
that in this case, on each coset $(a_1,a_2)+\p^2$ we have $xz_2\equiv\! 1\mod\p$. We change variable $v=xz_2-1\in\p$. Then
$A_2=\|y,v\|^s, B=\|u,y^3,y^2v\|^{-s}$, $C_2=\|y^4,yu,uv\|^{-s}$ and
$$J_2=q^{-1}\int_{u,y,v\in\p}|u|^{3s-4}|y|^2A_2BC_2d\mu.$$
Since $\mu\{(u,y,v)\in\p\mid v(u)=X,~v(y)=Y,~v(v)=Z \}=(1-q^{-1})^3q^{-X-Y-Z}$,
one has
\begin{align*}
J_2 =
q^{-1}(1-q^{-1})^3\sum_{(X,Y,Z)\in\p}&\left(q^{(-3s+3)X}q^{-3Y}q^{-Z}
q^{-s\min\{\min\{Y,Z\} }\right.\\
&\times\left. q^{s\min\{X,3Y,2Y+Z\}}q^{s\min\{X+Y,X+Z,4Y\}}\right).
\end{align*}
Again computing
$\sum_{(X,Y,Z)\in\NN^3}a^Xb^Yc^Zd^{-\min\{Y,Z\}}d^{\min\{X,3Y,2Y+Z\}}d^{\min\{X+Y,X+Z,4Y\}}$
and then applying for $a=q^{-3s+3},~b=q^{-3},~c=q^{-1}$ and $d=q^s$ yields $J_2$ which we 
record in the Appendix.

Summing up $\cal{Z}=\cal{Z}_1+\cal{Z}_2+\cal{Z}_3$, we obtain
$$\zeta_{\bf{H}(\o[x]/(x^3))}(s)=1+(1-q^{-1})^{-1}\cal{Z}=\frac{(1-t)(1-q^2t^2)(1-q^4t^3)}{(1-qt)(1-q^3t^2)(1-q^5t^3)}.$$
Hence
\begin{equation}\label{n=3}
\zeta_{\H(\O[x]/(x^3))}(s)=\frac{\zeta_K(s-1)}{\zeta_K(s)}\cdot\frac{\zeta_K(2s-3)}{\zeta_K(2s-2)}\cdot\frac{\zeta_K(3s-5)}{\zeta_K(3s-4)}. 
\end{equation}
\section{Open questions}
\subsection{Heisenberg group scheme}\label{sec questions Heisenberg}
Formulae \eqref{H over O},~\eqref{n=2} and \eqref{n=3} agree in the following conjectured formula.
\begin{conj}\label{conj Heisenberg over n}
The representation zeta function of $\H(\O[x]/(x^n))$ is
\begin{equation}\label{eq conj Hn}
\zeta_{\H(\O[x]/(x^n))}(s)=\prod_{i=1}^{n}\frac{\zeta_K(is-2i+1)}{\zeta_K(is-2i+2)}.
\end{equation}

\end{conj}
If  Conjecture~\ref{conj Heisenberg over n} holds, then the zeta functions $\zeta_{\H(\O[x]/(x^n))}(s)$ shares uniform analytic properties with $\zeta_{\H(\O)}(s)$ as
suggested by the following conjecture.
\begin{conj}\label{conj R}
Let $R$ be a ring which is torsion-free finitely generated over $\ZZ$. Then
the representation zeta function  $\zeta_{\H(R)}(s)$ has the following properties:
\begin{enumerate}
\item Its abscissa of convergence is $\a(\H)=2$.
\item It can be analytically continued to the whole complex plane. The continued zeta function has no singularities on the line
$\rea(s)=2$, apart from a simple pole at $s=\a(\H)$.
\end{enumerate}
\end{conj}
Recall that the representation zeta function of $\H(\ZZ_p)$ is
$$\zeta_{\H(\ZZ_p)}(s)= 1+(1-p^{-1})^{-1}\int_{\substack{x\in p\ZZ_p\\ y\in W_1(\ZZ_p)}}|x|^{s-2}d\mu=\frac{1-p^{-s}}{1-p^{1-s}}.$$

When computing the zeta function for $\H(\O_\p)$, one just needs to replace $p$ by
$q=|\O/\p|$ the residue field cardinality, $\ZZ_p$ by $\O_\p$, and replace $p$-adic norm $|.|$ by $\p$-adic norm $|.|_\p$  to
obtain
$$\zeta_{\H(\O_\p)}(s) = 1+(1-q^{-1})^{-1}\int_{\substack{x\in \p\O_\p\\ y\in W_1(\O_\p)}}|x|_\p^{s-2}d\mu=\frac{1-q^{-s}}{1-q^{1-s}}.$$
\begin{question}
Can one define the domain $W$, a \textit{valuation} and \textit{norm} $|.|_n$ on $\O_\p[x]/(x^n)$ compatible with the $\p$-adic
norm $|.|_\p$ such that the zeta function of $\H(\O_\p[x]/(x^n))$ can be computed as follows
$$\zeta_{\H(\O_\p[x]/(x^n))}(s)=1+(1-q^{-1})^{-1}\int_{W}|x|_n^{s-2}d\mu?$$ 
\end{question}

By expanding the conjectured formula \eqref{eq conj Hn} for the local zeta function, we get
\begin{equation}\label{eq conj Hn expand}
\zeta_{\H(\O_\p[x]/(x^n))}(s)=\sum_{I\subseteq[n-1]_0}f_I(q^{-1})\prod_{i\in I}\frac{q^{2n-2i-1-(n-i)s}}{1-q^{2n-2i-1-(n-i)s}},
\end{equation}
where $I$ runs over all subsets of $[n-1]_0:=\{0,1\cdots,n-1\}$ and $f_I(q^{-1})=(1-q^{-1})^{|I|}$. The formula \eqref{eq conj Hn expand}
looks similar to {\cite[(1.12)]{SV/14}}. However, we have been unable to mimic the inductive proof of {\cite[Theorem C]{SV/14}}
to yield \eqref{eq conj Hn expand}.

\subsection{Unipotent group schemes}\label{sec unipotent}
Once we understand the zeta function of $\H(\O_\p[x]/(x^n))$, 
the next step is the following.

Let $\Lambda$ be a finitely generated free and torsion-free $\O$-Lie
lattice of nilpotency class $c$ and $\O$-rank $h$. If $c>2$ we assume
that $\Lambda':=[\Lambda,\Lambda]\subseteq c!\Lambda$. This enables us
to associate to $\Lambda$ a unipotent group scheme $\G:=\G_\Lambda$
(cf.\ {\cite[Section 2.1.2]{SV/14}}). The group $\G(\O)$ is a $\cal{T}$-group
of nilpotency class $c$ and Hirsch length $h\cdot[K:\QQ]$.
The Heisenberg group scheme $\H$ is an example of such a unipotent group scheme.

The zeta function $\zeta_{\G(\O)}(s)$ is the Euler product
$$\zeta_{\G(\O)}(s)=\prod_{\p\in\Spec(\O)}\zeta_{\G(\O_\p)}(s)$$
ranging over all nonzero prime ideals $\p$ of $\O$; cf.~{\cite{SV/14}}. There exists a finite
set $S$ of prime ideals such that for each $\p\notin S$, the local zeta function $\zeta_{\G(\O_\p)}(s)$ is a rational function
in $q^{-s}$ and satisfies a functional equation upon inversion of $q$; see~{\cite[Theorem A]{SV/14}}.
Moreover, each  such local representation zeta function can be expressed in terms of a $\p$-adic integral; cf.~{\cite[Corollary 2.11]{SV/14}}.

It is tempting to investigate the zeta function of $\G(\O[x]/(x^n))$ when $n$ tends to $\infty$. As we have seen for the Heisenberg group scheme, if
Conjecture~\ref{conj Heisenberg over n} holds then
$$\varinjlim_n\zeta_{\H(\O[x]/(x^n))}(s)=\prod_{i=1}^\infty\frac{\zeta_K(is-2i+1)}{\zeta_K(is-2i+2)}.$$
It is natural to ask the following.
\begin{question}
Let $\G$ be a unipotent group scheme as above. 
\begin{enumerate}
\item Is $\G(\O[[x]])$  twist-rigid, that is, is $r_n(\G(\O[[x]]))$ finite for every $n\in\NN$? Does $\G(\O[[x]])$ have polynomial representation growth?
\item If (1) has positive answers then is it true that 
$$\zeta_{\G(\O[[x]])}(s):=
\sum_{n=1}^\infty r_n(\G(\O[[x]]))n^{-s}=\varinjlim_n\zeta_{\G(\O[x]/(x^n))}(s)?$$
\end{enumerate}
\end{question}

It is proved in {\cite[Theorem A]{DV/15}} that the zeta function $\zeta_{\G(\O)}(s)$
has rational abscissa of convergence $\a(\G)$, which is independent on the number field $K$,
and can be analytically continued to  $\rea(s)>\a(\G)-\delta(\G)$ for some $\delta(\G)>0$.
In the spirit of Conjecture~\ref{conj R}, we formulate the following.
\begin{conj}
Let $R$ be a ring which is finitely generated torsion-free $\O$-module. Then the 
representation zeta function of $\G(R)$ has the following properties:
\begin{enumerate}
\item The abscissa of convergence $\a(\G)$ of $\zeta_{\G(R)}(s)$ is independent of $R$.
\item It can be meromorphically continued to the half-plane $\rea(s)>\a(\G)-\delta(\G)$ for some $\delta(\G)>0$, where $\delta(\G)$ is independent of $K$.
\end{enumerate}
\end{conj}

\subsection{Topological representation zeta functions}
Topological zeta functions offer a way to define a limit as $p\to 1$ of families
of $p$-adic zeta functions. 
Let $\G$ be a unipotent group scheme defined as in Section~\ref{sec unipotent}.
In \cite{Rossmann/15}, Rossmann introduces and studies topological
representation zeta functions associated to unipotent  group schemes
$\G$. 
Informally, we define the topological representation zeta function
$\zeta_{\G,\rm{top}}(s)$ to be the constant term of $\zeta_{\G(\ZZ_p)}(s)$ as 
a series in $p-1$. 
\begin{example}
Consider the Heisenberg group scheme $\H$. Expanding $p^z=(1+(p-1))^z$ into a series in $p-1$, we obtain 
$\zeta_{\H(\ZZ_p)}(s)=\frac{s}{s-1}+O(p-1)$ and hence
$\zeta_{\H,\rm{top}}(s)=\frac{s}{s-1}$.

Let $\QQ[\varepsilon_n]=\QQ[x]/(x^n)$ and for a $\QQ$-algebra $\g$, let $\g[\varepsilon_n]=\g\otimes_\ZZ\QQ[\varepsilon_n]$
regarded as a $\g$-Lie lattice. Let $\H[\varepsilon_n]$ denote the group attached to $\g[\varepsilon_n]$; cf.~{\cite[Section~7]{Rossmann/15}}.
Then $\H[\varepsilon_n](\ZZ_p)=\H(\ZZ_p[x]/(x^n))$. Hence
\begin{align}
\zeta_{\H[\varepsilon_2],\rm{top}}(s)&=\frac{s(2s-2)}{(s-1)(2s-3)}=\frac{2s}{2s-3},\label{eq top}\\
\zeta_{\H[\varepsilon_3],\rm{top}}(s)&=\frac{s(2s-2)(3s-4)}{(s-1)(2s-3)(3s-5)}=\frac{2s(3s-4)}{(2s-3)(3s-5)}.\label{eq top2}
\end{align}
\end{example}
An algorithm to compute topological representation zeta functions is implemented in \cite{Rossmann-zeta}.
Notice that Rossmann's method computes topological representation zeta function directly without
first computing the corresponding $\p$-adic representation zeta function.
Formulae \eqref{eq top} and \eqref{eq top2} are consistent with computation results in \cite{Rossmann-zeta}. Notice also that \cite{Rossmann-zeta}
can only compute  the topological representation zeta function of $\H[\varepsilon_n]$ up to $n=3$ as we
have done here for their $\p$-adic representation zeta functions.  Conjecture~\ref{conj Heisenberg over n}
suggests the following analogue for topological representation zeta functions of $\H[\varepsilon_n]$.
\begin{conj}
The topological representation zeta function of $\H[\varepsilon_n]$ is
$$\zeta_{\H[\varepsilon_n],\mathrm{top}}(s)=\prod_{i=1}^n\frac{is-2i+2}{is-2i+1}.$$
\end{conj}
\begin{remark}
All Questions in {\cite[Section~7]{Rossmann/15}}, except Question $7.3$ which is not yet known, have positive answers for $\G=\H[\varepsilon_n]$ with $n\leq 3$.
\end{remark}

\section{Appendix}

\begin{tabular}{l|l}
$\cal{Z}_{322a}$   &  
	\begin{tabular}[x]{@{}l@{}}
	$(1-q^{-1})^4\left(
    \frac{q^{-2}t}{(1-q^{-2})(1-q^{-3})(1-qt)}+\frac{pt^2}{(1-q^{-3})(1-t)(1-qt)}+\frac{q^4t^3}{(1-t)(1-qt)(1-q^3t^2)}\right.$\\
    $+\frac{q^7t^5}{(1-qt)(1-q^3t^2)(1-q^4t^3)}+\frac{q^{10}t^7}{(1-q^3t^2)(1-q^4t^3)(1-q^6t^4)}+\frac{q^{13}t^9}{(1-q^4t^3)(1-q^9t^6)(1-q^6t^4)}$\\
    $\left.+\frac{q^{16}t^{12}}{(1-q^4t^3)(1-q^3t^3)(1-q^9t^6)}\right)$   
    \end{tabular}
\\
\\

$\mathcal{Z}_{322b1}$ & 
	\begin{tabular}[x]{@{}l@{}}
	$(1-q^{-1})^4\left( \frac{t}{(q-1)(1-q^{-3})(1-t)}+\frac{q^2t^2}{(1-q^{-1})(1-t)(1-q^3t^2)}
    +\frac{q^5t^4}{(1-q^{-1})(1-q^3t^2)(1-q^6t^4)}\right.$\\

	$+\frac{q^8t^6}{(1-q^{-1})(1-q^9t^6)(1-q^6t^4)}+\frac{q^{11}t^8}{(1-q^{-1})(1-q^9t^6)(1-q^2t^2)}+\frac{q^{14}t^{10}}{(1-q^9t^6)(1-q^5t^4)(1-q^2t^2)}$\\
	$\left.+\frac{q^{17}t^{12}}{(1-q^9t^6)(1-q^8t^6)(1-q^5t^4)}+\frac{q^{20}t^{15}}{(1-q^9t^6)(1-q^3t^3)(1-q^8t^6)}\right)$
	\end{tabular}
\\
\\

$J_1$ &
	\begin{tabular}[x]{@{}l@{}}
	$q^{-2}(1-q^{-1})^2\left(\frac{t}{(1-q^{-3})(1-t)}+\frac{q^3t^2}{(1-t)(1-q^3t^2)}
	+\frac{q^6t^4}{(1-q^6t^4)(1-q^3t^2)}+\frac{q^9t^6}{(1-q^9t^6)(1-q^6t^4)} 
	\right.$\\
	$+\left.\frac{q^{12}t^9}{(1-q^3t^3)(1-q^9t^6)}\right)$
	\end{tabular}
\\
\\

$J_2$ & 
	\begin{tabular}[x]{@{}l@{}}	
	$q^{-1}(1-q^{-1})^3\left( \frac{q^{-1}t}{(1-q^{-1})(1-q^{-4})(1-q^{-1}t)}
	+\frac{q^2t^2}{(1-q^{-1})(1-q^2t^2)(1-q^{-1}t)}\right.$\\
	$+\frac{q^5t^3}{(1-q^{-1})(1-q^5t^3)(1-q^2t^2)} 
	+\frac{q^8t^6}{(1-q^{-1})(1-q^3t^3)(1-q^5t^3)}
	+\frac{q^{-4}t}{(1-q^{-3})(1-q^{-4})(1-q^{-1}t)}$\\
	$+\frac{q^{-1}t^2}{(1-q^{-3})(1-q^{-1}t)(1-q^2t^2)}
	+\frac{q^2t^3}{(1-q^{-3})(1-q^5t^3)(1-q^2t^2)}
	+\frac{q^5t^4}{(1-q^{-3})(1-t)(1-q^5t^3)}$\\
	$+\frac{q^8t^5}{(1-q^5t^3)(1-q^3t^2)(1-t)}
	+\frac{q^{11}t^7}{(1-q^5t^3)(1-q^6t^4)(1-q^3t^2)}
	+\frac{q^{14}t^9}{(1-q^5t^3)(1-q^9t^6)(1-q^6t^4)}$\\
	$\left.+\frac{q^{17}t^{12}}{(1-q^5t^3)(1-q^9t^6)(1-q^6t^4)}\right)$
\end{tabular}
\\
\end{tabular}

\end{document}